 \documentclass[a4paper,Glenn,12pt,reqno]{amsart}
\usepackage[PetersLenny]{fncychap}
\usepackage[latin1]{inputenc} 
 \usepackage[T1]{fontenc} \usepackage[english ]{babel,minitoc}
 \usepackage[nice]{nicefrac} 
\usepackage{graphicx,fancyhdr,fancybox,rotating,xcolor,appendix}
\usepackage{ulem,hhline,dsfont,pifont,multicol,lmodern}
\usepackage{amsmath,amsthm,amsfonts,amsbsy,amssymb,geometry,times,pst-node}
\usepackage{mathpazo} \usepackage[pdftex]{hyperref}
 \numberwithin{equation}{section}  \makeatletter\@addtoreset{equation}{section}
\usepackage{euscript}
\usepackage{textcomp} \usepackage{sidecap} \usepackage{pdfpages} \usepackage{setspace}
 \newcommand{\scal}[1]{\left<#1\right>}
\newtheorem{theorem}{Theorem}[section]
\newtheorem{proposition}[theorem]{Proposition}

\newtheorem{lemma}[theorem]{Lemma}
\newtheorem{corollary}[theorem]{Corollary}

\newtheorem{remark}[theorem]{Remark}
\newcommand{\V}{\mathcal{V}}  
 \newcommand{\R}{\mathbb{R}}   \newcommand{\Rd}{\mathbb{R}^d}
 \newcommand{\C}{\mathbb{C}}   \newcommand{\Cd}{\mathbb{C}^{d}}
    	       \newcommand{\ldrd}{L^{2}(\Rd)}      
\newcommand{\Hzw}{\mathcal{H}(\C^2_{z,w})}
 \newcommand{\Lnur}{L^{2}(\R,\C)}
\newcommand{\bz}{\overline{z}}  \newcommand{\bw}{\overline{w}}
\newcommand{\bxi}{\overline{\xi}} \newcommand{\sxi}{{\xi^*}} \newcommand{\txi}{\widetilde{\xi}}

\newcommand{\gauss}{\mu}
\begin{document}

\title[Bivariate poly-analytic Hermite polynomials]{Bivariate poly-analytic Hermite polynomials}
\author{Allal Ghanmi}  
\author{Khalil Lamsaf}   
 \address{
Analysis, P.D.E. $\&$ Spectral Geometry, CeReMAR
\newline
Department of Mathematics, Faculty of Sciences,
Mohammed V University, P.O. Box 1014
Rabat, Morocco}

\date{\today}
\maketitle

\begin{abstract}
A new class of bivariate poly-analytic Hermite polynomials is considered.
We show that they are realizable as the Fourier-Wigner transform of the univariate complex Hermite functions and form a nontrivial orthogonal basis of the classical Hilbert space on the two-complex space with respect to the Gaussian measure. Their basic properties are discussed, such as their three term recurrence relations, operational realizations and  differential equations (Bochner's property) they obey.
 Different generating functions of exponential type are obtained. Integral and exponential operational representations are also derived.  
 Some applications in the context of  integral transforms and  the concrete spectral theory of specific magnetic Laplacians are discussed.
\end{abstract}

\section{Introduction} \label{s1}
The so-called univariate (poly-analytic) complex Hermite polynomials (UHCP), denoted  $H_{m,n}(z,\bz)$, constitute an orthogonal basis of the classical Hilbert space on the complex plane with respect to the Gaussian measure $e^{-|z|^2}dxdy$. They were introduced by It\^o \cite{Ito52} in the framework of complex Markov process and turned out to be useful in many different contexts. In fact, they have been used as a basic tool
in the study of, for instance, the nonlinear analysis of traveling wave tube amplifiers \cite{Barrett}, the spectral theory of some second order differential operators \cite{Shigekawa87,Matsumoto96,Gh2017Mehler}, the study of some special integral transforms \cite{IntInt06,BenahmadiG2019}, coherent states theory \cite{AliBagarelloHonnouvo10,AliBagarelloGazeau13}, combinatory \cite{Ismail13a,IsmailTrans2016} and signal processing \cite{RaichZhou04,DallingerRuotsalainenWichmanRupp10}.
For their basic properties and applications, one can refer to \cite{Gh13ITSF,IsmailTrans2016,DunklXu14,BenahmadiG2019,Gh2017Mehler}.

Bivariate complex polynomials of Hermite type can be defined in many different ways. The natural ones consist of considering the tensor product $H_{m}(z)H_{n}(w)$ of the univariate
holomorphic Hermite polynomials $H_{m}(z)$ or also by replacing $\bz$ in $H_{m,n}(z,\bz)$ by the variable $w$, leading to the two-variable holomorphic Hermite polynomials $H_{m,n}(z,w)$ considered in \cite{IsmailTrans2016}.
A systematic study of their analytic properties is presented in \cite{GorskaHorzelaSzafraniec2017}. See also \cite{Zhi-GuoLIU2017} for quite variant class in three variables $H_{m,n}(z,w,u)$. The $u$ variable can be seen as a physical parameter that interprets time or magnitude of a magnetic field \cite{Gh2017Mehler,BenahmadiG2019}.
The tensor product $H_{m,n}(z,\bz)H_{m',n'}(w,\bw)$ gives rise to another class of bivariate poly-analytic Hermite polynomials. 

In the present paper, we introduce a nontrivial  class of bivariate (poly-analytic) complex orthogonal polynomials. They are not a standard tensor product of the UHCP, but they are with special composition operators. More precisely, following the same scheme giving rise to the UCHP from the real Hermite polynomials via a like-binomial formula, we can suggest the following  
\begin{align}\label{HM}
H_{m,n,m',n'}(z,w) := H_{m,n}(z+iw,\bz-i\bw)  H_{m',n'}(\bz+i\bw,z-iw). 
\end{align}
 We will focus on their basic properties.  Mainly, we provide the corresponding creating and annihilating operators, three term recurrence formulas, Rodrigues type formula and special differential equations they obey. Connection to the UCHP is also given.
Moreover, different representations are derived such as the exponential operational representation and the integral representation by monomials. The realization as Fourier-Wigner transform of the UCHP is investigated. 
We also show that these polynomials form an orthogonal basis of $\Hzw:= L^2(\C^2,e^{-2(|z|^2+|w|^2)}d\lambda)$, the Hilbert space on the two-dimensional complex space with respect to the Gaussian measure.
Summation formulas including some generating functions are obtained. 
Interesting applications (which goes beyond the scope of this work), in the context of integral transforms and the concrete$L^2$- spectral analysis of special magnetic Laplacian, can be clearly stated. The concrete description of these points will be the subject of a forthcoming paper.

The basic topics that we need in developing these items are collected in Section 2. Thus, we begin by recalling some backgrounds, concerning  Fourier-Wigner transform and the univariate poly-analytic Hermite polynomials. 
Our main results are stated and proved in Section 3.
The last section is devoted to some concluding remarks concerning some direct applications.

\section{Backgrounds}

\subsection{Fourier-Wigner transform.} \label{s2}

It is defined as a bilinear mapping
on $\ldrd \times \ldrd$ by
\begin{align}\label{Exp:Vs}
\V_d(f,g)(p,q)=\left(\frac{1}{2\pi}\right)^{\frac{d}{2}}
\int_{\Rd} e^{i \scal{y,q}}f\left(y+\frac{p}{2}\right)\overline{g\left(y-\frac{p}{2}\right)}dy
\end{align}
for every $(p,q)\in \Rd \times \Rd$.
It is an important tool in various fields of research like harmonic analysis, signal analysis, engineering, and the physical sciences. In fact, it is essential in studying Weyl transform \cite{Folland1989,Thangavelu,Wong1998} and in interpreting quantum mechanics as a form of nondeterministic statical dynamics \cite{Moyal}.
The Fourier-Wigner transform $\V_d$ preserves the tensor product 
\begin{align}\label{ProductFormula}
\V_d (f, g)(p, q) = \prod_{j=1}^n \V_1(f_j , g_j)(p_j , q_j),
\end{align}
for given $f_j , g_j \in L^2(\R)$; $j = 1, \cdots , n$, 
where $\V_1$, in the right hand-side, denotes the one-dimensional Fourier-Wigner transform. Moreover, it satisfies the Moyal formula
\begin{align}\label{Moyal}
\scal{\V_d(f,g),\V_d(\varphi,\psi)}_{L^{2}(\Cd)}= \scal{f,\varphi}_{\ldrd} \scal{\psi,g}_{\ldrd}.
\end{align}
This Moyal property and the action of Fourier-Wigner transform $\V_1$ on the classical (physicist) univariate real Hermite functions 
$$h^{real}_{n}(x)=e^{-\frac{x^{2}}{2}}H^{real}_{n}(x) =  (-1)^{n}e^{\frac{x^{2}}{2}}\dfrac{d^{n}}{dx^{n}}(e^{-x^{2}})$$
are fundamental tools in reproving the known fact that the UCHP constitute an orthogonal basis of the Hilbert space $L^{2}(\C,e^{-|z|^2}dxdy)$ (\cite{Ito52,IntInt06,ABEG2015}).
In fact, we have (\cite[Theorem 3.1]{ABEG2015}) 
\begin{align}\label{FWTHmn}
H_{m,n}\left(z ,\bz \right)= (-1)^{n}\frac{\sqrt{2}}{\sqrt{2}^{m+n}} e^{\frac{|z|^{2}}{2}} \V_1(h^{real}_{m},h^{real}_{n})(\sqrt{2}x,\sqrt{2}y).
\end{align}
Accordingly, it is rather natural to consider the set of functions $\V_2(h_{m,n},h_{m',n'})$, where $h_{m,n}\left(z,\bz\right):= e^{-|z|^2/2}H_{m,n}\left(z,\bz\right)$ denote the Hermite functions  associated to the UCHP, and to look for their basic properties and  explicit expression. These aims are the subject of Section 3. 
To this end, we collect below the basic properties of the UCHP that we need to develop the rest of this paper.

\subsection{The univariate poly-analytic Hermite polynomials.} The orthogonal UCHP are defined by Rodrigues formula
\begin{align}\label{chp}
H_{m,n} (z,\bz )=(-1)^{m+n}e^{ |z|^2 }\dfrac{\partial ^{m+n}}{\partial \bz^{m} \partial z^{n}} \left(e^{- |z|^2 }\right) 
\end{align}
and satisfies
\begin{align}\label{orthoHmn} \int_{\C} H_{m,n}(z,\bz ) H_{j,k}(z,\bz ) e^{-|z|^2} d\lambda(z) = \pi m!n! \delta_{m,n}.
\end{align}
Their expression in terms of the generalized Laguerre polynomials 
is given in \cite[Eq. (2.3)]{IntInt06},  
while the one in terms of the 
univariate real Hermite polynomials $h^{real}_m$
is given by (\cite{Gh13ITSF}, see also \cite{IsmailTrans2016})
\begin{align}\label{Hmnreal}
H_{m,n}\left(z,\bz\right)
=\left(\dfrac{1}{2}\right)^{m+n}m!n! \sum_{j=0}^{m}\sum_{k=0}^{n} 
\frac{(-1)^{k} (i)^{j+k}}{j!k!}   \frac{H^{real}_{m+n-j-k}(x) H^{real}_{j+k}(y)}{(m-j)!(n-k)!}
\end{align}
with $z=x+iy$; $x,y\in \R$.
The corresponding exponential operational formula 
\begin{align}\label{HmnDelta}
H_{m,n}(z,\bar z) = e^{- \Delta_\C}\left( z^m \bz ^n \right), \quad \Delta_\C := \frac{\partial^{2}}{\partial z\partial\overline{z}},
\end{align}
is proved in \cite[Theorem 2.1]{IsmailTrans2016}.
Added to the integral representation \eqref{FWTHmn} via the Fourier--Wigner transform, such polynomials  obey  (\cite[Theorem 2.4]{BenahmadiG2019})
\begin{align}\label{intHermite}
H_{m,n}(z;\bz) =  \frac{\mu (-\alpha)^m(\beta)^n}{\pi}  
e^{|z|^2} \int_{\C} \xi^m \overline{\xi}^n e^{  -\gauss|\xi|^2 +\alpha \scal{\xi,z} - \beta \overline{\scal{\xi,z}}} d\lambda(\xi).
\end{align}
Here $\alpha,\beta$ are complex numbers such  that $\alpha\beta = \gauss >0$.
By taking for example $\gauss=1$ and $\alpha =-\beta=i$, the integral representation \eqref{intHermite} reduces further to the one obtained by Ismail \cite[Theorem 5.1]{IsmailTrans2016}. 

\subsection{Generating and bilinear generating functions.} 
The considered polynomials can be defined equivalently by means of the generating function
\begin{align}\label{GenHmn}
\sum_{m=0}^{+\infty}\sum_{n=0}^{+\infty}\frac{u^{m}}{m!}\frac{v^{n}}{n!}H_{m,n}(z,\overline{z})
= e^{ - uv + z u + \overline{z} v}.
\end{align}
Added to \eqref{GenHmn}, the polynomials $H_{m,n}$ satisfy further interesting (partial) generating functions \cite{BenahmadiG2019}. 
We collect here some bilateral generating functions of Mehler type that generalize, somehow, the classical Poisson kernel for the real Hermite polynomials $H_m^{real}(x)$ (see e.g. \cite{Mehler1866,Rainville71,Andrews}).
The following 
\begin{align}\label{genfct1hh}
\sum\limits_{n=0}^{+\infty} \frac{ t^n }{n!}  H_{m,n}(z,\bz ) H_{n,m'}(w,\bw )  &=
(-t)^{m'}   H_{m,m'}( z -tw, \bz - \overline{t}\bw) e^{ t w\bz },
\end{align}
valid for every $t$ in the unit circle and $z,w\in \C$, is proved in \cite[Theorem 3.1]{BenahmadiG2019}, generalizing the one in \cite[Proposition 3.6]{Gh13ITSF}.
However, they are three widest generalizations of \eqref{genfct1hh}. The first one asserts that the quantity 
$$E(u,v|z,w):= \sum_{m=0}^{+\infty}\sum_{n=0}^{+\infty}  \frac{u^mv^n}{m!n!} H_{m,n}(z,\bz) H_{m,n}(w,\bw)$$
is given by the closed formula
\begin{align}\label{Mehler2}
\frac{1}{1 - uv}  \exp\left(- \frac{ uv( |z|^2 + |w|^2)  - uzw - v\overline{z}\overline{w} }{1 -  uv} \right)
\end{align}
for every $u,v \in \C$ such that $|uv| <1$.
This is exactly Mehler's formula for $H_{m,n}(z;\bz)$ given by W\"unsche \cite{Wunsche1999} without proof and recovered by Ismail \cite[Theorem 3.3]{IsmailTrans2016} as a specific case of his Kibble-Slepian formula \cite[Theorem 1.1]{IsmailTrans2016} (see \cite[Theorem 4.1]{Gh2017Mehler} for a special generalization).
As consequence of \eqref{Mehler2} combined with the integral representation \eqref{intHermite}, one can derive an interesting self-reciprocity property 
(see \cite[Theorem 4.2]{Gh2017Mehler}).
The closed explicit expression of the Heat kernel function for a spacial magnetic Laplacian is also given in \cite[Theorem 3.3]{Gh2017Mehler}.
The extension to $t$ in the unit circle and $|u|<1$ is given by \cite[Theorem 2.4]{BenahmadiG2019}
\begin{align} \label{BilGen2}
E(u,t|z,\bw) &= \frac{1}{(1-  tu)} \exp\left(\frac{- tu |z -tw|^2}{1-  tu}\right) e^{  t w\bz }
\end{align}
as well as (\cite[Theorem 2.3]{BenahmadiG2019})
\begin{align}\label{BilGen1}
\sum_{m=0}^{+\infty}\sum_{n=0}^{+\infty} \frac{u^m t^n }{ m!n!}  H_{m,n}(z,\bz ) H_{n,m'}(w,\bw )
&=  ( \bw -t\bz + u )^{m'} e^{t \bz w - ut ( w-\overline{t} z) }.
\end{align}
Remarkable applications are given in the context of integral transforms  connecting  
the generalized Bargmann--Fock spaces or also $L^2(\C,e^{-  |z|^2}dxdy)$ to the two-dimensional Bargmann-Fock space $\mathcal{F}^{2}(\C^2)$ are given (see \cite{BenahmadiG2019} for detail).

\section{Bivariate complex Hermite polynomials}

The polynomials $H_{m,n,m',n'}$ defined through $\eqref{HM}$, depending in $z$, $w$ and in  their complex conjugates, are called here bivariate complex poly-analytic Hermite polynomials and abbreviated as BCPHP.
Throughout this paper, we will use simultaneously the notation $	H_{m,n,m',n'}(z,w)$ as well as the short one
$H_{M}(Z,\overline{Z})$, where $M$ is a $4$-uplet and $Z=(z,w)$ with $ \overline{Z}=(\bz,\bw)$, to design the BCPHP. We will also make use of the multi-index notation. For given $M=(m_{1},m_{2},m'_{1},m'_{2})$ and $N= (n_{1},n_{2},n'_{1},n'_{2})$, we define
$|M|= m_{1}+m_{2}+m'_{1}+m'_{2}$ and the Kr\"onecker symbol $\delta_{M,N} :=
\delta_{m_{1},n_{1}}\delta_{m_{2},n_{2}}\delta_{m'_{1},n'_{1}}\delta_{m'_{2},n'_{2}}$, while $M! :=  m_{1}!m_{2}!m'_{1}!m'_{2}!$. The binomial coefficient, for $N\leq M$ (i.e., $m_\ell\leq n_\ell$ and $m'_\ell\leq n'_\ell$; $\ell=1,2$), is defined by
$$ \binom{M}{N}=\prod_{\ell=1,2}\binom{m_\ell}{n_\ell} \binom{m'_\ell}{n'_\ell} .$$

It is clear from \eqref{HM} that the $H_{m,n,m',n'}(z,w)$ are polynomials in $z+iw$ of degree $m$, in $\bz-i\bw$ of degree $n$, in $\bz+i\bw$ of degree $m'$ and in $z-iw$ of degree $n'$. They are also polynomials in $z,\bz,w$ and $\bw$ of degrees $m+n$; $m'+n'$, $m+n$ and $m'+n'$, respectively. 
 As particular case, we have $H_{m,n,0,0}(z,w) =  H_{m,n}(z+iw,\bz-i\bw)$, $H_{m,0,m',0} = (z+iw)^{m} (\bz+i\bw)^{m'} $ and $H_{m,0,0,n'} = (z+iw)^m (z-iw)^{n'}$.
  Moreover, we have the following symmetry relationships for indexes by taking the complex conjugation
 \begin{align}\label{barconj}
  \overline{H_{m,n,m',n'}(z,w)}  =   H_{n,m,n',m'}(z,w) = H_{m,n,m',n'}(\bz,-\bw)
   \end{align}
   and 
     \begin{align}
  H_{m,n,m',n'}(\bz,\bw) =  H_{m',n',m,n}(z,w). 
   \label{Hmnm'n'(barz,barw)}
  \end{align}

  It becomes clear from above the need of associating to each $(z,w)\in \C^2$ the complex numbers $\xi = z + i w$, $\sxi= \bz + i \bw $, $\bxi = \bz - i \bw$ and $\txi=  z - i w$. Notice that $\bxi$ and $\txi$ are respectively the complex conjugates of $\xi$ and $\sxi$. 
 It should be mentioned here that the operations ${^{-}}$, ${^*}$ and $\widetilde{\quad}$ are involutions, i.e., $ \sxi^* = \widetilde{\txi}  = \overline{\bxi} =\xi$. Moreover, they are pairwise commuting, in the sense that $ \bxi^* = \overline{\sxi} = \txi$, $ \widetilde{\bxi} = \overline{\txi} = \sxi$ and  $ \txi^* = \widetilde{\sxi} = \bxi$. Thus, the components 
  $z$ and $w$, in the definition of $\xi$, are reals if and only if $ \bxi = \sxi$ and $ \txi = \xi$.

  \subsection{Orthogonality.}
The first main result of this subsection shows that the polynomials $H_M$ are orthogonal in $\Hzw$.
  
  \begin{theorem}\label{Thm:Orth}
  	We have 
  	\begin{equation}\label{orthogH}
  	\int_{\C^2}H_{M}(Z,\overline{Z})H_{N}(\overline{Z},Z)
  	e^{-2|Z|^2}d\lambda(Z) =  
  	\frac{\pi^{2}}4 M! \delta_{M,N}.
  	\end{equation}
  \end{theorem}
  
  \begin{proof}
  	Denote the left hand-side of \eqref{orthogH} by $I_{M,N}$ and notice that the Lebesgue measure on $\C^2=\R^4$, satisfies $4d\lambda(Z)=4d\lambda(z,w) =d\lambda(\xi,\txi) $. Therefore, using the fact $2|Z|^2=2(|z|^2+|w|^2)= |\xi|^2+|\txi|^2$ and Fubuni theorem, we obtain
  	 	\begin{align*}
 	I_{M,N} 	
 	&=\frac 14  \scal{ H_{m_{1},m_{2}}\, , \,  H_{n_{1},n_{2}} }_{L^2(\C_\xi,e^{-|\xi|^2}d\lambda)}  \scal{  H_{m'_{1},m'_{2}} \, , \,  H_{n'_{1},n'_{2}} }_{L^2(\C_\xi,e^{-|\txi|^2}d\lambda)}  ,
  	\end{align*}
  	where $M=(m_1,m_2,m'_1,m'_2)$ and $N=(n_1,n_2,n'_1,n'_2)$. 	 
  	The result follows making use of the orthogonality property \eqref{orthoHmn} for the UCHP.
  \end{proof}

  \begin{theorem}\label{Thm:basis}
  	The bivariate complex Hermite polynomials $H_{m,n,m',n'}$ form an orthogonal basis for the Hilbert space $\Hzw$.
  \end{theorem}
  
  \begin{proof} We need only to prove completeness. Let $f\in \Hzw $ such that 
  	$$ 	\int_{\C^2} f(z,w) H_{m,n,m',n'}(z,w) e^{-2(|z|^2+|w|^2)} d\lambda(z,w) =0 $$
  	for every $4$-uplet $M=(m,n,m',n')$. Therefore, we get 
  	$$ \int_{\C} \left( \int_{\C} f (\xi,\txi) H_{m,n}(\xi,\bxi) e^{-|\xi|^2}
  	d\lambda(\xi) \right) H_{m',n'}(\sxi,\txi) e^{-|\txi|^2}
  	d\lambda(\txi) = 0. $$
  	This implies that the function $ \psi: \txi \longmapsto  \int_{\C} f (\xi,\txi) H_{m,n}(\xi,\bxi) e^{-|\xi|^2}
  	d\lambda(\xi)$ vanishes a.e. on $\C$, for $\psi \in L^2(\C,e^{-|\txi|^2}d\lambda(\xi))$ and the UCHP being a basis of it. The same argument shows that $f (\xi,\txi)=0$ a.e. on $\C^2$. This completes the proof.
  \end{proof}

  	\subsection{Rodrigues type formula.}
  	In our framework, any function $f$ 
  	in $(z,w)\in \C^2$ can also be seen as a function in  $\xi = z + i w$ and $\txi=  z - i w$ (and implicitly in $\bxi = \bz - i \bw$ and $\sxi= \bz + i \bw$).
  	We perform the following family of first order differential operators
  	\begin{align} \label{Ader12}
  	& A_{\xi}  := \frac{1}{2}\left( \frac{\partial}{\partial z}   - i \frac{\partial}{\partial w}  \right) =: \frac{\partial}{\partial \xi}
  	\quad \mbox{and} \quad A_{\sxi} := \frac{1}{2}\left( \frac{\partial}{\partial \bz}   - i \frac{\partial}{\partial \bw}  \right) =: \frac{\partial}{\partial \sxi}
  	\end{align}
  	and their like-complex conjugates
  	\begin{align}\label{Ader34}
  	& A_{\bxi} := \frac{1}{2}\left( \frac{\partial}{\partial \bz} + i \frac{\partial}{\partial \bw} \right) =: \frac{\partial}{\partial \bxi} 
  	\quad \mbox{and} \quad A_{\txi} := \frac{1}{2}\left( \frac{\partial}{\partial z} + i \frac{\partial}{\partial w} \right) =: \frac{\partial}{\partial \txi} .
  	\end{align}
  	
  	The notation of derivation in the right hand-sides of equalities \eqref{Ader12} and \eqref{Ader34} is justified by the following lemma, whose proof is straightforward.
  	
  	\begin{lemma}\label{propertiesofA}
  		The operators $A_{\xi} , A_{\bxi}, A_{\sxi}$ and $A_{\txi}$ are pairwise commuting. 
  		Moreover, they look like the derivation operators with respect to $\xi,\bxi,\sxi$ and $\txi$ (seen as independent variables), respectively. More precisely, we have
  		\begin{align*} 
  		& A_{\xi}  (\xi^{m} \bxi^{n}\sxi^{j} \txi^{k})  = m \xi^{m-1} \bxi^{n} \sxi^{j} \txi^{k}; \qquad 
  		A_{\bxi} (\xi^{m} \bxi^{n}\sxi^{j} \txi^{k})  = n \xi^{m} \bxi^{n-1}\sxi^{j} \txi^{k}; \\
  		& A_{\sxi} (\xi^{m} \bxi^{n}\sxi^{j} \txi^{k}) = j \xi^{m} \bxi^{n}\sxi^{j-1} \txi^{k}; \qquad
  		A_{\txi} (\xi^{m} \bxi^{n}\sxi^{j} \txi^{k})  = k \xi^{m} \bxi^{n}\sxi^{j} \txi^{k-1};\\
  		& A_{\xi}(\bxi^{n}\sxi^{j} \txi^{k})  = A_{\bxi}(\xi^{m} \sxi^{j} \txi^{k}) =  A_{\sxi}(\xi^{m} \bxi^{n} \txi^{k}) =  A_{\txi}(\xi^{m} \bxi^{n}\sxi^{j}) = 0.
  		\end{align*}
  	\end{lemma}
  	
  	Accordingly, it is essential to make the following observation. 
  	
  	\begin{lemma}\label{Lem:xiHmn}
  		We have 
     	\begin{align*}
      	(-1)^{m}  e^{|\xi|^2} \left(  A_{\bxi}^m(
  		\bxi^n e^{-|\xi|^2 }) \right) = H_{m,n}(\xi,\bxi)
  		\end{align*}
  		and
  		\begin{align*}
  		(-1)^{m'}	\left( A_{\txi}^{m'} (  \txi^{n'} e^{- |\sxi|^2}) \right)H_{m',n'}(\sxi,\txi).
  		\end{align*}
  	\end{lemma}
  	
  	\begin{proof}
  	  		The result can be handled by direct computation using Leibniz formula for the operator $ A_{\bxi}^m$ and $A_{\txi}^{m'}$ combined with appropriate change of variable. However, the proof, we adopt here, follows easily making use of the commutation rules and their explicit actions on the monomials (Lemma \ref{propertiesofA}) that we may extend to  Gaussian functions. 
  	\end{proof}
    		  	
Subsequently we assert the following.
  	
  	\begin{theorem}\label{Thm:tensor}
  		Keep notations as above. The bivariate complex Hermite polynomials  $	H_{m,n,m',n'}$
  		satisfy  
  		\begin{align}\label{DefBCHP} H_{m,n,m',n'}(z,w) : =  (-1)^{m+n+m'+n'} e^{|\xi|^2 + |\sxi|^2}  A_{\bxi}^m A_{\xi}^n  A_{\txi}^{m'} A_{\sxi}^{n'} \left(e^{-|\xi|^2 - |\sxi|^2} \right).
  		\end{align}
  	\end{theorem}
  	
  	\begin{proof}
  		Starting from \eqref{HM} and making use of Lemma \ref{Lem:xiHmn}, we obtain
  		  		\begin{align*}
  		H_{m,n,m',n'}(z,w)  &=  H_{m,n}(\xi,\bxi)H_{m',n'}(\sxi,\txi)
  		\\ 	&= (-1)^{m+m'}  e^{|\xi|^2 + |\sxi|^2} \left(  A_{\bxi}^m(
  		\bxi^n e^{-|\xi|^2 }) \right)  
  		\left( A_{\txi}^{m'} (  \txi^{n'} e^{- |\sxi|^2}) \right)\\
 	&=  (-1)^{m+n+m'+n'} e^{|\xi|^2 + |\sxi|^2}  A_{\bxi}^m A_{\xi}^n  A_{\txi}^{m'} A_{\sxi}^{n'} \left(e^{-|\xi|^2 - |\sxi|^2} \right).
  		\end{align*}
  		This completes the proof of \eqref{DefBCHP}.  	
  	\end{proof}
  	
  	\subsection{Creating and annihilating operators.}
  	The three term recurrence formulas 
  	\begin{align*}
  	& H_{m+1,n,m',n'}      =   \xi H_{m,n,m',n'}  - A_{\bxi} H_{m,n,m',n'}  \\
  	& H_{m,n+1,m',n'}      =   \bxi H_{m,n,m',n'}- A_{\xi} H_{m,n+1,m',n'} \\
  	&H_{m,n,m'+1,n'}    =   \sxi H_{m,n,m',n'} -  A_{\txi} H_{m,n,m'+1,n'}\\
  	&  H_{m,n,m',n'+1}   =   \txi H_{m,n,m',n'} - A_{\sxi}H_{m,n,m',n'}
  	\end{align*}
  	can be derived easily from Rodrigues' type formula \eqref{DefBCHP} (or also from their analogues for the UCHP) combined with 
  	Lemma \ref{propertiesofA}.
  	Accordingly, the operators $\xi- A_{\bxi} $, $\bxi -A_{\xi}$, $\txi-A_{\sxi} $ and $\sxi -A_{\txi}$
  	are raising operators for the polynomials $H_{m,n,m',n'}(z,w)$, in the sense that
  	\begin{align}
  	& (\xi- A_{\bxi}) H_{m,n,m',n'} = H_{m+1,n,m',n'}  \label{newarizing1}
  	\\	& (\bxi -A_{\xi}) H_{m,n,m',n'} = H_{m,n+1,m',n'} \label{newarizing2} \\
  	& (\txi-A_{\sxi} ) H_{m,n,m',n'} = H_{m,n,m'+1,n'} \label{newarizing3} \\
  	& (\sxi -A_{\txi}) H_{m,n,m',n'} = H_{m,n,m',n'+1}. \label{newarizing4}
  	\end{align}
  	
  	The following realizations of the $H_{m,n,m',n'} $ are of particular interest.
  	
  	\begin{proposition}\label{Prop:realiz}
  		We have
  		\begin{align}
  		H_{m,n,m',n'}   &=(\sxi-A_{\txi})^{m'} (\txi-A_{\sxi})^{n'}  H_{m,n,0;0} \label{newrealiz1a}\\
  		&=(\xi-A_{\bxi})^m  (\bxi-A_{\xi})^n H_{0,0,m',n'} \label{newrealiz2a} \\
  		& =(\xi-A_{\bxi})^m  (\bxi-A_{\xi})^n (\sxi -iA_{\txi})^{m'} (\txi-A_{\sxi})^{n'} \cdot(1)  \label{newrealiz3}
  		\end{align}
  		as well as
  		\begin{align}
  		H_{m,n,m',n'}  & =(\xi -A_{\bxi})^m  (\sxi-A_{\txi})^{m'}  (\bxi^{n} \txi^{n'} )  \label{newrealiz2}\\
  		& =(\bxi-A_{\xi})^{n} (\txi-A_{\sxi})^{n'}  (\xi^m \sxi^{m'}) . \label{newrealiz1}
  		\end{align}
  	  	\end{proposition}

  \begin{proof} 
   The proof of \eqref{newrealiz3} is a immediate consequence of \eqref{newrealiz1a} and \eqref{newrealiz2a} since $ H_{0,0,0,0}= 1$.
  	While the ones of the first and the second assertion are similar.
  	 	The proof of \eqref{newrealiz1a} readily follows using
  	$H_{m,n,m',n'}  = (\xi-A_{\bxi})^m H_{0,n,m',n'}  $ as well as $H_{m,n,m',n'}= (\bxi-A_{\xi})^n H_{m,0,m',n'} $,
  	which follow by induction from the three terms recurrence formulas \eqref{newarizing1} and \eqref{newarizing2}.  	   	
  By means of similar arguments, we have
  	\begin{align*}
  	H_{m,n,m',n'} =   (\xi-A_{\bxi})^{m} (\sxi-A_{\txi})^{m'}  H_{0,n,0,n'}.
  	\end{align*}
  	This infers the identity \eqref{newrealiz2} since $H_{0,n,0,n'} = \bxi^n \txi^{n'}$. The proof of the last one, \eqref{newrealiz1}, is similar. However, it can be obtained  by the complex conjugation  thanks to \eqref{barconj} and the identities
  	$\overline{\bxi} =\xi$, $\overline{\sxi} = \txi$ and $\overline{\txi} = \sxi$.
   	\end{proof} 
  		
  	\begin{remark} Notice that the operator $\txi-A_{\sxi} $ commutes with $\bxi-A_{\xi} $ and $\xi-A_{\bxi}$ but not with $\sxi-A_{\txi} $. This fact leads to other possible forms of the polynomials $H_{m,n,m',n'}$. For example, among others we have
  		\begin{align}
  		H_{m,n,m',n'}(z,w)   &=(\sxi-A_{\txi})^{m'} (\txi-A_{\sxi})^{n'} (\xi-A_{\bxi})^m  (\bxi-A_{\xi})^n \cdot(1)  \\ \label{realiz4}
  		& =(\xi-A_{\bxi})^m(\sxi-A_{\txi})^{m'} (\bxi-A_{\xi})^n (\txi-A_{\sxi})^{n'} \cdot(1).
  		\end{align}
  	\end{remark}

  	As immediate consequence of Proposition \ref{Prop:realiz}, we show that $A_{\xi}$, $A_{\bxi}$, $A_{\sxi}$ and $A_{\txi}$ are lowering operators for the $ H_{m,n,m',n'}$. More exactly, we assert the following.
  	
  	\begin{proposition}\label{Prop:raising}  We have
  		\begin{align}
  		A_{\xi} H_{m,n,m',n'}   & =  m  H_{m-1,n,m',n'}  \label{lowering1}\\
  		A_{\bxi} H_{m,n,m',n'}  & =  n  H_{m,n-1,m',n'}  \label{lowering2}\\
  		A_{\sxi} H_{m,n,m',n'}  & =  m'  H_{m,n,m'-1,n'} \label{lowering3} \\
  		A_{\txi} H_{m,n,m',n'}  & =  n'  H_{m,n,m',n'-1} .\label{lowering4}
  		\end{align}
  	\end{proposition}
  	
  	\begin{proof}
  		The identity \eqref{lowering2} (resp. \eqref{lowering4}) follows easily from \eqref{newrealiz1} (resp. \eqref{newrealiz2}) by acting on by the operator $A_{\bxi}$ (resp. $A_{\txi}$) keeping in mind the identity $A_{\bxi} (\xi^m \bxi^n) = n \xi^{m} \bxi^{n-1}$ (resp. $A_{\txi} (\sxi^{m'} \txi^{n'} ) = n' \sxi^{m'} \txi^{n'-1}$). While \eqref{lowering1} (resp. \eqref{lowering3}) is the complex conjugate of \eqref{lowering2} (resp. \eqref{lowering4})  thanks to \eqref{barconj}.
  	\end{proof}

	\subsection{Exponential operational representation.}
This result establishes an analog of the operational formula \eqref{HmnDelta}. 
To this end, we denote by 
$$\Delta_{\C^2} := \frac{\partial^2}{\partial z \partial \bz} + \frac{\partial^2}{\partial w\partial \bw} 
$$
the Laplace-Beltarmi operator on $\C^2$. 
It should be noticed here that 
$$ A_{\xi} A_{\bxi} = \frac 14 \Delta_{\C^2} + \frac i4 \Box_{\C^2}$$
 and 
 $$ A_{\sxi} A_{\txi} = \frac 14 \Delta_{\C^2} - \frac i4 \Box_{\C^2},$$
 so that 
 $$A_{\xi} A_{\bxi} + A_{\sxi} A_{\txi}= \frac{1}{2} \Delta_{\C^2}. $$
  Above $\Box_{\C^2}$ denotes the second order differential operator given by
$$ \Box_{\C^2} :=  \frac{\partial^2}{\partial z \partial \bw} - \frac{\partial^2}{\partial w\partial \bz} 
.$$

\begin{theorem}\label{Thm:ExpRep}
	For every $(z,w)\in\C^2$, we have the following operational formula
	\begin{equation}\label{O.F}
	H_{m,m',n,n'}(z,w) =  e^{-( A_{\xi} A_{\bxi} + A_{\sxi} A_{\txi})} \left(\xi^{m}\bxi^{n}\sxi^{m'}\txi^{n'}\right).
	\end{equation}
\end{theorem}

\begin{proof}
	The result readily follows using \eqref{HmnDelta}. Indeed, by seen $\xi$ and $\sxi$ as complex variables with complex conjugates $\bxi$ and $\txi$, respectively, we can write  
	\begin{align*}
	H_{m,m',n,n'}(z,w) & = H_{m,n}(\xi,\bxi) H_{m',n'}(\sxi,\txi)
\\& =  e^{- A_{\xi} A_{\bxi}}\left( \xi^m \bxi^n \right) 
e^{-A_{\sxi} A_{\txi}}\left( \sxi^{m'} \txi^{n'} \right) 
		\\&=  e^{- ( A_{\xi} A_{\bxi} + A_{\sxi} A_{\txi} ) }\left( \xi^m \bxi^n \sxi^{m'} \txi^{n'} \right).
	\end{align*}
	Thus, one concludes for \eqref{O.F} since $ A_{\xi} A_{\bxi}$ and $A_{\sxi} A_{\txi}$ commute.
\end{proof}

\begin{remark}
	The result of Theorem \ref{Thm:ExpRep} can also be handled making use of Theorem \ref{Thm:sumtensor} below. 
\end{remark}

\subsection{Special second differential equations.}

Using the introduced lowering and raising operators, it is easy to see that the polynomials $H_{m,n,m',n'}$ solve some second order differential equations. Indeed, by applying $\xi-A_{\bxi} $ (resp. $\bxi-A_{\xi}$, $\sxi-A_{\txi}$ and $ \txi-A_{\sxi}$) to \eqref{lowering1} (resp. \eqref{lowering2}, \eqref{lowering3} and \eqref{lowering4}),
it is easy to check the following result.

\begin{lemma}\label{lem:Diffeq}
	 The polynomials $H_{m,n,m',n'}$ satisfy Bochner's property for being eigenfunctions of the second order differential operator $ L_{\xi}:=\xi A_{\xi}-A_{\xi} A_{\bxi} $ (resp. $L_{\bxi}:=\bxi A_{\bxi}-A_{\xi} A_{\bxi}  $, $L_{\sxi}:=\sxi A_{\sxi}-A_{\sxi} A_{\txi} $ and $L_{\txi} := \txi A_{\txi}-A_{\sxi} A_{\txi}$) with $m$ (resp. $n$, $m'$ and $n'$) as corresponding eigenvalue.
\end{lemma}

Consequently, the polynomials $H_{m,n,m',n'}$ are clearly solutions of the eigenvalue and the commune eigenvalue problems for the second order differential operators
\begin{align}
& S_{\xi}  : =  -\frac 14\bigg( \Delta_{\C^2} + i\Box_{\C^2} - 2(E_{z} + E_{w}) +2i( F_{zw} - F_{wz}) \bigg)  \label{lap1} \\
& S_{\sxi} : = -\frac 14\bigg( \Delta_{\C^2} -i\Box_{\C^2} -2(E_{\bz} + E_{\bw}) +2i ( F_{\bz\bw} - F_{\bw\bz}) \bigg) , \label{lap2} 
\end{align}
and their complex conjugates $S_{\bxi}$ and $S_{\txi}$. Here
$E_{z} := z\frac{\partial}{\partial z}$ and its complex conjugate $E_{\bz}:= \bz\frac{\partial}{\partial \bz} $ are the usual  Euler operators on $\C$, and 
$F_{zw}$, $F_{z\bw}$, $F_{\,\bz w}$, $F_{\,\bz\,\bw}$,
$F_{wz}$, $F_{w\bz}$, $F_{\bw z}$ and $F_{\bw\bz}$, are the coupled Euler operators defined by
$$F_{uv}  := u \frac{\partial}{\partial v}.$$
For proving the previous assertion, it suffices to make the following observation.  

\begin{lemma}\label{lem:Lap} 
	The operators $S_{\xi}$, $S_{\bxi} $, $S_{\sxi} $ and $S_{\txi}$ are exactly those involved in 
	Lemma \ref{lem:Diffeq}, $S_{\bullet}=L_{\bullet}$.
\end{lemma}

Moreover, the $L^2$-eigenvalue problem associated to the special Landau Hamiltonian $L_\xi:= \xi A_{\xi}-A_{\xi} A_{\bxi}$ acting on $\Hzw$ can explicitly described. Namely, we prove the following

\begin{theorem}
	The spectrum of the operator $S_\xi$ (resp. $ S{\bxi}$, $S_{\sxi} $ and $S_{\txi}$) acting on $\Hzw$ is purely discrete and is constituted of the eigenvalues $\ell =0,1,2, \cdots$.
\end{theorem}

\begin{proof}
	Lemmas \ref{lem:Diffeq} and \ref{lem:Lap} show that the polynomials $H_{m,n,m',n'}$ are eigenfunctions of $ L_{\xi}$ (resp. $ L_{\bxi}$, $L_{\sxi} $ and $L_{\txi}$). Moreover, they belong to the Hilbert space $\Hzw$, thanks Theorem \ref{Thm:Orth}, with $m$ as corresponding eigenvalue (Lemma \ref{lem:Diffeq}). 
	Therefore, one may conclude in virtue of their completeness proved in Theorem \ref{Thm:basis}.	
\end{proof}

 \subsection{Connection to the UCHP.} 

 Added to \eqref{HM}, the polynomials $H_{M}$ are closely connected to the univariate complex Hermite polynomials $H_{j,k}$ by means of the so-called $(4,2)$-binomial operator $\mathcal{B}^{(4,2)}_M$, defined on double function sequences $f=(f_{m,n})_{m,n}$ on the complex plane by 
	\begin{align*}
\mathcal{B}^{(4,2)}_M (f)(z,w) := 
\sum_{J=0}^{M}  (-1)^{k+k'} i^{|J|} \binom{M}{J}  
f_{m+n'-j-k',m'+n-j'-k}(z,\bz)   f_{j+k',j'+k}(w,\bw)   
\end{align*}
with $M=(m,n,m',n')$ and $J=(j,k,j',k')$.
 
\begin{theorem}\label{Thm:sumtensor}
	Set $H^{complex}:=( H_{m,n})_{m,n}$. Then, we have 
	\begin{align}\label{Hmncomp1}
	H_{m,n,m',n'}\left(z, w\right)
	&= \frac{1}{\sqrt{2}^{|M|}}  \mathcal{B}^{(4,2)}_M (H^{complex})(\sqrt{2}z,\sqrt{2}w) .
	\end{align} 
\end{theorem}

\begin{proof}
Direct computation shows that the quantity 
$$Q^a_{m,n,m',n'}(z,w) :=  e^{-a\Delta_{\C^2}} \left((z+iw)^{m}(\bz-i\bw)^{n} (\bz+i\bw)^{m'}(z-iw)^{n'}\right),$$ for any positive real number $a>0$, is equal to
		\begin{align*} 
Q^a_{m,n,m',n'}(z,w) 
	&= \mathcal{B}^{(4,2)}_M (e^{-a \Delta_\C}(E) )(z,w),
\end{align*}
	where $E$ stands for the sequence $(e_{m,n})_{m,n}$ with $e_{m,n}(z,\bz):= z^m\bz^n$.
Therefore, using the fact that $$ e^{-a \Delta_\zeta}(\zeta^m\overline{\zeta}^n) = \sqrt{a}^{m+n} H_{m,n}\left( \frac{\zeta}{\sqrt{a}}, \frac{\overline{\zeta}}{\sqrt{a}}\right) ,$$
we get 
	\begin{align*} 
Q^a_{m,n,m',n'}(z,w) = \sqrt{a}^{|M|} \mathcal{B}^{(4,2)}_M (H^{complex} ) \left( \frac{\zeta}{\sqrt{a}}, \frac{\overline{\zeta}}{\sqrt{a}}\right).
\end{align*}
Finally,  the identity \eqref{Hmncomp1} follows by taking $a=1/2$ and making use of Theorem \ref{Thm:ExpRep}, since $ Q^{1/2}_{m,n,m',n'}=H_{m,n,m',n'}$.
\end{proof}

\begin{remark}
The result of Theorem \ref{Thm:sumtensor} can also be obtained starting from Theorem \ref{Thm:tensor} and making use of the binomial identity for commuting operators 
and  Rodrigues' formula \eqref{chp} for the UCHP. 	
\end{remark}


\begin{remark} By taking $m'=n'=0$ and replacing $w$ by $-iw$ in \eqref{Hmncomp1}, we recover Runge formula for the UCHP \cite[Proposition 3.8]{Gh13ITSF},
		\begin{align}\label{Runge2013}
	H_{m,n}\left(z+w, \bz+\bw \right)
	= \frac{m!n!}{\sqrt{2}^{m+n}} 
		\sum_{j=0}^{m}	\sum_{k=0}^{n}   \frac{	H_{m-j-k',n-k}(\sqrt{2}z,\sqrt{2}\bz)}{(m-j)(n-k)j} \frac{H_{j ,k}(\sqrt{2} w,\sqrt{2}\bw)}{j!k!}  .
		\end{align} 
\end{remark}

A linearization formula for the product $H_{m,n}(z, \bz) H_{m',n'}(z,\bz) $ is proved by Ismail in \cite[Theorem 4.1]{IsmailTrans2016}, Theorem \ref{Thm:sumtensor} furnish another linearization formula for such product. 

\begin{corollary}
	We have 
		\begin{align} \label{productCHPmm}
	H_{m,n}(z, \bz) H_{m',n'}(\bz, z)&=  	\frac{1}{\sqrt{2}^{|M|}} 
	\sum_{J=0; \ j+k'=j'+k}^{M}  (-1)^{j+k} (j+k')! i^{|J|} \binom{M}{J} 
	\\&\qquad \times
	H_{m+n'-j-k',n+m'-j-k'}(\sqrt{2}z,\sqrt{2}\bz)  . \nonumber 
	\end{align}
\end{corollary}

\begin{proof}
It is immediate taking $w=0$  in Theorem \ref{Thm:sumtensor} and using the fact that
$H_{m,n,m',n'}(z,0)= H_{m,n}(z, \bz) H_{n',m'}(z,\bz)$ as well as $H_{r,s}(0,0)= (-1)^r r!$ when $r=s$ and $H_{r,s}(0,0)= 0$ otherwise.   
\end{proof}

\subsection{Integral representations.}
The first integral representation for 
$H_{M}(Z,\overline{Z})$ follows immediately from their definition and the integral representation \eqref{intHermite} of $H_{m,n}(z,\bz)$.
 Namely, if for given complex numbers $\alpha,\alpha',\beta,\beta'$ such that $\mu:=\alpha \beta >0$ and $\mu':=\alpha' \beta' >0$, we let $E^{\alpha',\beta'}_{\alpha,\beta}(u,v|X,Y)$ denote the exponential function
 \begin{align}\label{DefEkernel}
 E^{\alpha',\beta'}_{\alpha,\beta}(u,v|X,Y) :=
 \exp	\left( -\mu|u|^2 -\mu'|v|^2 +\alpha \scal{u,X} - \beta \overline{\scal{u,X}} +\alpha' \scal{v,Y} - \beta' \overline{\scal{v,Y}} \right) 
 \end{align}
 then the following holds true 
 \begin{align}\label{IntRep0}
 H_{m,n,m',n'}(z,w)  &=  \frac{(-1)^{m+m'}\mu\mu' \alpha^m\beta^n(\alpha')^{m'}(\beta')^{n'}}{\pi^2} e^{2( |z|^2 + |w|^2)} 
 \\&\times \int_{\C^2} u^m \overline{u}^n   v^{m'} \overline{v}^{n'}
 E^{\alpha',\beta'}_{\alpha,\beta}(u,v|z+iw,\bz+i\bw)
 d\lambda(u,v).
 \nonumber 
 \end{align}
 
A variant integral representation of \eqref{IntRep0}, involving the auxiliary variables 
$\xi_{u,v}:= u + iv$, $\overline{\xi_{u,v}} := \overline{u} -i\overline{v}$, $\xi_{u,v}^* := \overline{u} + i\overline{v}$ and $\widetilde{\xi_{u,v}} := u -iv$
in the integrand, is the following.  

\begin{theorem}
	For every $\alpha,\beta\in\C$ such that $\mu=\alpha\beta>0$, we have
		\begin{align}\label{IntRep}
 H_{m,n,m',n'}\left( \frac{z}{\sqrt{2}} , \frac{w}{\sqrt{2}}\right)   &=	 \frac{\mu^2(-\alpha)^{m+n'}{\beta}^{m'+n}}{\pi^2 \sqrt{2}^{|M|}}
	e^{  |z|^2+ |w|^2}
	\\& \times\int_{\C^2} (u+iv)^m (\overline{u}-i\overline{v})^{n} (\overline{u}+i\overline{v})^{m'} (u-iv)^{n'} 
	E^{\alpha,\beta}_{\alpha,\beta}(u,v|z,w) d\lambda(u,v).\nonumber
	\end{align}	 
\end{theorem}

\begin{proof}
	For the proof, we make use of Theorem \ref{Thm:sumtensor} 
	 and the integral representation of the UCHP given through \eqref{intHermite}. 
	Indeed, the quantity $ H_{m,n,m',n'}\left( \frac{z}{\sqrt{2}} , \frac{w}{\sqrt{2}}\right)  $ can be rewritten as 
$$\frac{\mu^2(-\alpha)^{m+n'}{\beta}^{m'+n}}{\pi^2 \sqrt{2}^{|M|}}
e^{ |z|^2+|w|^2}
\int_{\C^2}  S_M(u,v)  	E^{\alpha,\beta}_{\alpha,\beta}(u,v|z,w) d\lambda(u,v),$$
	where $S_M(u,v)$ is given by 
	\begin{align*}
	S_M(u,v) &:=  \mathcal{B}^{(4,2)}_M (E)(u,v)
	= (u+iv)^m (\overline{u}-i\overline{v})^n(\overline{u}+i\overline{v})^{m'}(u-iv)^{n'}.
		\end{align*} 
	This proves \eqref{IntRep}.
	\end{proof}

\begin{remark}
	For the particular case of $\alpha= -\beta =i$, the identity \eqref{IntRep} reads simply 
	\begin{align}\label{IntReppc}
H_{m,n,m',n'}\left( \frac{z}{\sqrt{2}} , \frac{w}{\sqrt{2}}\right)  &=\frac{(-i)^{|M|}}{\pi^2 \sqrt{2}^{|M|}} 
e^{|z|^2+ |w|^2}
\\& \times\int_{\C^2} \xi_{u,v}^m  {\overline{\xi_{u,v}}}^n{\xi_{u,v}^*}^{m'} {\widetilde{\xi_{u,v}}}^{n'}  
e^{ -|u|^2 -|v|^2 +2i\mathfrak{Re}( \scal{u,z}  +\scal{v,w} )}  d\lambda(u,v).
\nonumber
\end{align}
	\end{remark}


\subsection{Realization as Fourier-Wigner transform of the UCHP.}
The next result gives another integral representation of the bivariate complex Hermite polynomials by means of the UCHP. To this end, we consider the two-dimensional Fourier-Wigner transform $\V_2$ ($d=2$) and we set
$$ \V( \psi , \varphi )(z;w) := \V_2(\psi,\varphi)((x_{1},x_{2});(y_{1}, y_{2}))$$
for given complex variables $z_\ell=x_\ell+iy_\ell$; $\ell=1,2$, with $x_\ell,y_\ell\in \R$, and $ \psi , \varphi\in L^2(\R^2)$. 

\begin{theorem}\label{Thm:BchpFW}
	Let $h_{m,n}\left(z,\bz\right):= e^{-\frac{|z|^2}{2}}H_{m,n}\left(z,\bz\right)$ denote the complex Hermite functions. Then, we have 
	\begin{align} \label{BchpFW}    
 H_{m,n,m',n'}(z,w) = 2 (-1)^{m'+n}   
	e^{|z|^2+|w|^2} \V(h_{m,n'},h_{m',n})(2z,2w) .
	\end{align}
\end{theorem}

\begin{proof}
	Direct computation using \eqref{Hmnreal} 
	and the product formula \eqref{ProductFormula}, shows that for $M=(m,n,m',n') $ we have 
	\begin{align*} 
	\V_2(h_{m,n'},  h_{m',n})(\sqrt{2}z,\sqrt{2}w) 
&= \dfrac{1}{2^{|M|}}     \mathcal{B}^{(4,2)}_M (V)(\sqrt{2}z,\sqrt{2}w) ,
\end{align*}
	where we have set
$ 
	V:=(V_{r,s})_{r,s} =  \V_1(h^{real}_{r}, h^{real}_{s} ) .
$ 
But, in view of \eqref{FWTHmn}, we
 arrive at
	\begin{align*} 
\V_2(h_{m,n'},  h_{m',n})(\sqrt{2}z,\sqrt{2}w) 
& = \frac{ (-1)^{m'+n}}{2\sqrt{2}^{|M|}}
 e^{-\frac 12(|z|^2+|w|^2)} 	\mathcal{B}^{(4,2)}_M (H^{complex})  \left( z,w\right)
 \\& = \frac{ (-1)^{m'+n}}{2}  e^{-\frac 12(|z|^2+|w|^2)} H_M\left( \frac{z}{\sqrt{2}},\frac{w}{\sqrt{2}}\right)
 .
	\end{align*}
This establishes \eqref{BchpFW} thanks to Theorem \ref{Thm:sumtensor}.
\end{proof}

  \begin{remark}
	Orthogonality (Theorem \ref{Thm:Orth}) and completeness (Theorem \ref{Thm:basis}) of $ H_{M}(Z,\overline{Z})$ in the Hilbert space $\Hzw $
	can be recovered via the realization of these polynomials as Fourier-Wigner transform of the UCHP and using Moyal identity.
\end{remark}

The previous result implicitly states that the bivariate complex Hermite polynomials are expressible as finite sum of the Fourier-Wigner transform of the tensor product
\begin{align}\label{PrRealsh}
h^{real}_{m}\otimes h^{real}_{m'}(t,t') :=h^{real}_{m}(t)h^{real}_{m'}(t').
\end{align}
The next one gives a direct representation of $H_{m,n,m',n'}$ as Fourier-Wigner transform of such tensor product. Thus, we set 
	\begin{align} \label{DefVmnmn}
	V^{m,n}_{m',n'} := \V_2(h^{real}_{m}\otimes h^{real}_{m'},h^{real}_{n}\otimes h^{real}_{n'}) .
		\end{align}

\begin{theorem}\label{Thm:BchpFW2}
	We have 
		\begin{align} \label{BchpFW2}    
	H_{m,n,m',n'}\left(\frac{z}{\sqrt{2}},\frac{w}{\sqrt{2}} \right) = \frac{ (-1)^{n+n'} }{\sqrt{2}^{|M|-2}}  e^{\frac 12(|z|^2+|w|^2)} 
		V^{m,n}_{m',n'} (z+iw,\bz+i\bw).
	\end{align}
\end{theorem}

\begin{proof}
	Starting from \eqref{DefVmnmn} and using the product formula \eqref{ProductFormula} as well as \eqref{FWTHmn}, one gets
	\begin{align*} 
	V^{m,n}_{m',n'}(\xi,\sxi)
	&=\V_1(h^{real}_{m},h^{real}_{n})(\xi) \V_1( h^{real}_{m'}, h^{real}_{n'})(\sxi)\\
	&= (-1)^{n+n'} \frac{\sqrt{2}^{|M|}}{2} e^{-\frac 14(|\xi|^2+|\sxi|^2)} H_{m,n}\left(\frac{\xi}{\sqrt{2}},\frac{\bxi}{\sqrt{2}} \right) 
	H_{m',n'}\left(\frac{\sxi}{\sqrt{2}},\frac{\overline{\sxi}}{\sqrt{2}} \right).
	\end{align*}
	By taking $\xi=z+iw$, $\bxi=\bz-i\bw$, $\sxi=\bz+i\bw$ and $ \overline{\sxi} = z-iw =\txi$ for given $z,w\in\C$, we obtain
	 	\begin{align*} 
	 V^{m,n}_{m',n'}(\xi,\sxi)
	 &= (-1)^{n+n'} \frac{\sqrt{2}^{|M|}}{2}  e^{-\frac 14(|\xi|^2+|\sxi|^2)} H_{m,n,m',n'}\left(\frac{z}{\sqrt{2}},\frac{w}{\sqrt{2}} \right).
	 \end{align*} 
	 This is exactly \eqref{Thm:BchpFW2}
\end{proof}


\subsection{Exponential generating function.}

In this section, we investigate some basic generating functions. The few first ones follow from Mehler formulas for the UCHP presented in Subsection 2.3.

\begin{proposition} \label{prop:GenHM1}
	We have 
	\begin{equation}\label{gf-4}
	\sum_{n=0}^{+\infty}\frac{t^{n}}{n!}H_{m,n,n,m'}(z,w)=(-t)^{-m'} H_{m,m'}( \zeta_{t,z,w} ,\overline{\zeta_{t,z,w}} )e^{t(\bz^{2}+\bw^{2})}.	
	\end{equation}
	where $\zeta_{t,z,w}:= (z-t\bz)+i(w-t\bw)$.
\end{proposition}  

\begin{proof}
From \eqref{genfct1hh}, one deduces
\begin{align*}
\sum_{n=0}^{+\infty}\frac{t^{n}}{n!}H_{m,n,n,m'}(z,w)& =	\sum_{n=0}^{+\infty}\frac{t^{n}}{n!}H_{m,n}(\xi,\bxi)H_{n,m'}(\sxi,\txi)
\\&
=(-t)^{m'}H_{m,m'}(\xi-t\sxi,\bxi-\overline{t}\txi)e^{t\sxi\bxi}.
\end{align*}
This yields \eqref{gf-4} since $\sxi\bxi=\bz^{2}+\bw^{2}$ and $\xi-t\sxi=(z-t\bz)+i(w-t\bw)=\zeta_{t,z,w}$.
\end{proof}

\begin{remark}
	The particular case of  $t=-1$ reduces further to
	\begin{equation*}
\sum_{n=0}^{+\infty}\frac{(-1)^{n}}{n!}H_{m,n,n,m'}(z,w)=H_{m,m'}(2\zeta_{-1,z,w},2\overline{\zeta_{-1,z,w}}) e^{-(\bz^{2}+\bw^{2})}.
\end{equation*}
	with $\zeta_{-1,z,w} =\mathfrak{Re}(z)+i\mathfrak{Re}(w)$.
\end{remark}

The next one concerns the generating function
	\begin{equation*}
G_2(u,v|z,w):= \sum_{m=0}^{+\infty}	\sum_{n=0}^{+\infty}\frac{u^{m}v^{n}}{m!n!}H_{m,n,m,n}(z,w).
\end{equation*}

\begin{proposition}\label{prop:GenHM2}
	For any $u,v\in\C$ such $|uv|< 1$, we have 
	\begin{equation*}\label{GenFct3}
G_2(u,v|z,w)
	=\frac{1}{1-uv}\exp\left(
	\frac{(u+v-2uv)|z|^{2}- (u+v+2uv)|w|^2 +2i\mathfrak{Re}(z\bw)}{1-uv}\right).
	\end{equation*} 
\end{proposition} 

\begin{proof}
	The proof is straightforward and follows easily
	making use of \eqref{Mehler2}. 
\end{proof} 

\begin{proposition}\label{prop:GenHM3}
For $|u|<1$ and $|t|=1$, we have 
	\begin{equation}\label{GenFct4}
	\sum_{m=0}^{+\infty}\sum_{n=0}^{+\infty}
	\frac{u^{m}t^{n}}{m!n!}H_{m,n,n,m}(z,w)
	=\frac{e^{t(\bz^2+\bw^2)}}{1-ut}\exp\left(-\frac{ut}{1-ut}\left|z-t\bz +i(w-t\bw)\right|^2\right).
	\end{equation}
	In particular, we have 
	\begin{align}\label{GenFct4pc}
		\sum_{m=0}^{+\infty}\sum_{n=0}^{+\infty}
		\frac{u^{m}}{m!n!}H_{m,n,n,m}(z,w)
	=\frac{e^{\bz^2+\bw^2}}{1-u}\exp\left(-\frac{4u}{1-u}(\mathfrak{Im}(z)^2+\mathfrak{Im}(w)^2)\right)
	\end{align}
	for every $u\in\C$ such that $|u|< 1$.
\end{proposition} 

\begin{proof}
	The proof of \eqref{GenFct4} is immediate in view of from \eqref{BilGen2}. While the generating function \eqref{GenFct4pc} is in fact a particular case with $t=1$.
\end{proof}

 \begin{proposition}\label{prop:GenHM4}
 	For $|u|<1$ and $|t|=1$, we have 
 \begin{equation}\label{GenFct5}
\sum_{m=0}^{+\infty}\sum_{n=0}^{+\infty}\frac{u^{m}v^{n}}{m!n!}H_{m,n,n,m'}(z,w)=(tu+\zeta_{t,z,w})^{m'}\exp\left( t(\bz^2+\bw^2)+u\overline{\zeta_{t,z,w}}\right).
\end{equation}
 \end{proposition}

 \begin{proof}
 The direct use of \eqref{BilGen1}, that we rewrite in the form
 	\begin{align*}
 	\sum_{m,n=0}^{+\infty}\frac{u^{m}t^{n}}{m!n!}H_{m,n,n,m'}(z,w)
 	=(t\xi - t\bxi +tu)^{m'}e^{t\bxi\sxi+u(\xi-t\sxi)},
 	\end{align*}
 	yields \eqref{GenFct5}, since $\sxi\bxi=\bz^2+\bw^2$ and  $\txi-t\bxi=(z-t\bz) -i(w-t\bw)= \zeta_{t,z,w}$.
 \end{proof}

The next result concerns the special sum 
\begin{align}\label{GenFctHq1sum}
G_4(z,w|u,v,u',v'):= \sum_{m=0}^{+\infty}\sum_{n=0}^{+\infty}\sum_{m'=0}^{+\infty} \sum_{n'=0}^{+\infty} \frac{u^{m}}{m!}\frac{v^{n}}{n!} \frac{{u'}^{m'}}{m'!}\frac{{v'}^{n'}}{n'!}  H_{m,m',n,n'}(z,w).
\end{align}

\begin{proposition}\label{prop:GenHM5}
	The function  $G(z,w|u,v,u',v')$ is given by 
	\begin{align}\label{GenFctHq1}
	G_4(z,w|u,v,u',v')
	=e^{- uv - {u'}{v'}}
	e^{ z (u+v') + \bz (v+u') + iw(u-v') + i\bw(u'-v)}
	\end{align}
	for every $z,w,u,v,u',v'\in\C$.
\end{proposition}

\begin{proof}
	The closed expression of $G_4(z,w|u,v,u',v')$ is immediate from the definition of the BCPHP  and using the classical generating function \eqref{GenHmn} for the UCHP.
\end{proof}

The last results of this section are partial generating functions for $H_{M}$.

\begin{proposition}\label{Prop:PartialGF}
	For any $z,w,u,v,u',v'\in\C$, we have
	\begin{equation}\label{PartialGF1}
	\sum_{m=0}^{+\infty}	\sum_{m'=0}^{+\infty} \frac{u^{m}}{m!}\frac{u'^{m'}}{m'!}H_{m,n,m',n'}(z,w) 
	=(\bz-i\bw-u)^{n}(z-i w-u')^{n'}e^{u(z+iw)+{u'}(\bz + i\bw)}.
	\end{equation}
	and 
	\begin{equation}\label{PartialGF2}
\sum_{m=0}^{+\infty}	\sum_{n=0}^{+\infty} \frac{u^{m}}{m!}\frac{v^{n}}{n!}H_{m,n,m',n'}(z,w) 
= e^{u v - u(z+iw) - v(\bz - i\bw)}  H_{m',n'}(\bz+i\bw,w-iw).
\end{equation}
\end{proposition}

\begin{proof}
By rewriting the left hand-side in \eqref{PartialGF1} as
\begin{equation*}
\left( \sum_{m=0}^{+\infty}	 \frac{u^{m}}{m!} H_{m,n}(z+iw,\bz-i\bw) \right)
\left( 	\sum_{m'=0}^{+\infty} \frac{u'^{m'}}{m'!}H_{m',n'}(\bz+i\bw,w-iw)\right)  
\end{equation*}
and next applying the identity \cite[Proposition 3.4, Eq. (3.13)]{Gh13ITSF}
$$ \sum_{m=0}^{+\infty}	 \frac{u^{m}}{m!} H_{m,n}(\zeta,\overline{\zeta}) = (\overline{\zeta}-u) e^{u\zeta},$$
we get the identity \eqref{PartialGF1}. While \eqref{PartialGF2} is immediate from \eqref{GenHmn} 
since the left hand-side takes the form
$$ 
\left( \sum_{m=0}^{+\infty}	\sum_{n=0}^{+\infty} \frac{u^{m}}{m!}\frac{v^{n}}{n!}  H_{m,n}(z+iw,\bz-i\bw) \right)
 H_{m',n'}(\bz+i\bw,w-iw).$$ 
\end{proof}

\begin{remark}
		Similar partial generating functions to the ones in Proposition \ref{Prop:PartialGF} can be obtained by means of the symmetry properties \eqref{barconj} and \eqref{Hmnm'n'(barz,barw)}.
\end{remark}

%
%

\section{Concluding remarks}



Further many interesting algebraic properties of the introduced polynomials can be derived as immediate consequences of onces obtained in Section 3. The simply ones, like the generating functions in Propositions \ref{prop:GenHM1}--\ref{prop:GenHM5}, follow directly from different Mehler formulas for the UCHP. A Runge addition type formula for the polynomials $H_M$ 
\begin{equation}\label{Runge}
H_{m,n,m',n'}(z+z',w+w')= \frac{1}{\sqrt{2}^{|M|}}\sum_{J=0}^{M}\binom{M}{J}H_{J}(\sqrt{2}z,\sqrt{2}w)H_{M-J}(\sqrt{2}z',\sqrt{2}w').
\end{equation}
may be proved easily by means of \eqref{Runge2013}. While the use of the Nielson identity for the UCHP (\cite[Proposition 3.2]{Gh13ITSF})
$$ 
H_{m+p,n+q}\left(z,\overline{z}\right) = m!n!p!q! \sum_{j=0}^{m\wedge q} \sum_{k=0}^{n\wedge p}\frac{(-1)^{j+k}}{j!k!} \frac{H_{m-j,n-k}(z,\bz)}{(m-j) !(n-k)!} \frac{H_{p-k, q-j}(z, \overline{z})}{(p-k) !(q-j) !},
$$ 
where $m\wedge q:=min(m,q)$, 
gives rise to the following quadratic recurrence formula
\begin{equation}\label{Quadratic}
H_{M+N}(z,w)=\sum_{J=0}^{M\wedge N^{t}}(-1)^{|J|}\binom{M}{J}\binom{N}{J^{t}}H_{M-J}(z,w)H_{N-J^{t}}(z,w)
\end{equation}
for given $M=(m,n,m',n')$ and $N=(p,q,p',q')$,
with $M\wedge N^{t}:=(m\wedge q,n\wedge p, m'\wedge q', n'\wedge p')$. Here $J^{t}$ is defined by $J^{t}=(k,j,k',j')$ for $J=(j,k,j',k')$.
Notice also that the closed formula of the following special generating function for the UCHP.
\begin{equation}
T_M(u,v|z,w) := \sum_{j=0}^{\infty}\sum_{k=0}^{\infty}\sum_{j'=0}^{\infty}\sum_{k'=0}^{\infty}\frac{(-1)^{k'}i^{j}u^{j}v^{j'}z^{k}w^{k'}}{j!k!j'!k'!}H_{j'+m', k'+n'}(\bz,z)H_{j+m, k+n}(w,\bw) 
\end{equation}
 readily follows by rewriting $H_{m,n,m',n'}$ in the form
\begin{align*}
H_{m,n,m',n'}(z,w)=H_{m,n}(iw-(-z),-i\bw-(-\bz))H_{m',n'}(\bz-(-i\bw),z-iw)
\end{align*}
and next applying 
\cite[Theorem 4.10]{IsmailTrans2016}
\begin{equation*}
\sum_{j=0}^{\infty}\sum_{k=0}^{\infty} H_{j+m, k+n}(z,\overline{z}) \frac{u^{j} v^{k}}{j ! k !}=(-1)^{m+n} e^{u z+v \overline{z}-u v} H_{m, n}(z-v, \overline{z}-\overline{v}),
\end{equation*}
keeping in mind the fact that $H_{m,n}(iz,-i\bz)=i^{m+n}H_{m,n}(z,\bz)$. In fact, we assert   
	\begin{equation}
	T_M(u,v|z,w) = (-1)^{|M|}i^{m"+n'}e^{u\xi+v\sxi} H_{m,n,m',n'}(z,w).
	\end{equation}

Some analytic aspects of $H_M$ are encoded in Theorems \ref{Thm:Orth}, \ref{Thm:basis}, \ref{Thm:BchpFW} and  \ref{Thm:BchpFW2}. For example, Theorem \ref{Thm:basis} shows in particular that the Hilbert space $\Hzw$ possesses different  special $L^2$-Hilbertian orthogonal decompositions in terms of new functional poly-analytic Hilbert spaces of Bargmann type spanned by the polynomials $ H_{m,n,j,k}(z,w)$. 
In fact, by considering for example
$$ \mathcal{A}^{2}_{n,j,k}(\C^2) =  \overline{Span \{H_{m,n,j,k}(z,w); \, m=0,1,\cdots \} }^{\Hzw},$$
for fixed $n,j,k$, we claim that  $\mathcal{A}^{2}_{n,j,k}(\C^2)$ is a Hilbert subspace of $\Hzw$.
The special case $\mathcal{A}^{2}_{0,0,0}(\C^2)$ is realized as $\mathcal{A}^{2}_{0,0,0}(\C^2)=\ker(A_{\bxi}) \cap \ker(A_{\sxi}) \cap  \ker(A_{\txi})\cap \Hzw$ and therefore is contained in the two-dimensional Bargmann--Fock space $\mathcal{F}^{2}(\C^2)$ of $L^2$-holomorphic functions on $\C^2$ and coincides with the phase space 
\begin{equation}\label{Image}
\mathcal{A}^{2}(\C^2) := \left\{F\in \mathcal{F}^{2}(\C^2); \, \left(\frac{\partial}{\partial z} + i \frac{\partial}{\partial w}\right) F =0 \right\}  
\end{equation}
which is unitary isomorphic to the configuration space $\Lnur$ by means of the integral transform 
\begin{equation}\label{Ga}
\mathcal{G}^\nu f(z,w) 
:= \left(\frac{c_\nu^2}{\pi}\right)^{\frac{1}{2}}  \int_\R f(x)  e^{-   \left(x  -  \frac{z+iw}{2} \right)^2} dx
\end{equation}
obtained as the composition of $1d$- and $2d$-Segal--Bargmann transforms (see \cite{BDG2018}).
%
The generalization of the considered transform, to the context of the phase spaces $ \mathcal{A}^{2}_{n,j,k}(\C^2)$, can be constructed as a coherent state transform from $\Lnur$ onto  $ \mathcal{A}^{2}_{n,j,k}(\C^2)$. The associated kernel function is closely connected to the bilinear generating function
$$ K_{n,j,k}(t;z,w):=    \sum_{m=0}^{\infty} \frac{t^n H_m(x) H_{m,n,j,k}(z,w)}{m!}  .$$
%

Another interesting class of functional spaces in $\Hzw$ are the ones defined by
$$ \mathcal{A}^{2}_{m}(\C^2) =  \overline{Span \{H_{m,n,j,k}(z,w); \, \, n,j,k=0,1,\cdots \} }^{\Hzw},$$
leading to  another orthogonal Hilbertian decomposition of $\Hzw$. 
We  
claim that for every fixed $m$, the space $\mathcal{A}^{2}_{m}(\C^2) $ is closely connected to the concrete spectral analysis of the second order differential operator $S_\xi$ in \eqref{lap1} acting on $\Hzw$. In fact, one can show that $\mathcal{A}^{2}_{m}(\C^2) $ is an $L^2$-eigenspace of $S_\xi$ with $m$ as associated eigenvalue, 
$ 
\mathcal{A}^{2}_{m}(\C^2) =  \ker (S_\xi -m Id)|_{\Hzw} . 
$ 

Added to $\mathcal{A}^{2}_{m,n,j}(\C^2)$, $\mathcal{A}^{2}_{m}(\C^2)$ and their variants, one has to define in a similar way the spaces  $\mathcal{A}^{2}_{m,n}(\C^2)$ (as well as its variants). It should be noted here that the following local orthogonal Hilbertian decompositions hold trues
$$
\mathcal{A}^{2}_{m}(\C^2) = \bigoplus_{n=0}^\infty \mathcal{A}^{2}_{m,n}(\C^2) 
$$
and 
$$\mathcal{A}^{2}_{m,n}(\C^2) =\bigoplus_{j=0}^\infty  \mathcal{A}^{2}_{m,n,j}(\C^2) . $$
Therefore, the global ones for $\Hzw$ are the following
$$ \Hzw = \bigoplus_{m=0}^\infty \mathcal{A}^{2}_{m}(\C^2) = \bigoplus_{m,n=0}^\infty \mathcal{A}^{2}_{m,n}(\C^2)  =\bigoplus_{m,n,j=0}^\infty  \mathcal{A}^{2}_{m,n,j}(\C^2) .$$
The concrete description of these spaces is the subject of a forthcoming paper.

\end{document}